\documentclass{amsart}

\usepackage{tikz}
\usetikzlibrary{arrows}
\usetikzlibrary{decorations.markings}

\usepackage[all]{xypic}
\usepackage{epsf}
\usepackage{verbatim} 
\usepackage{amsmath}
\usepackage{amsfonts}
\usepackage{amssymb}
\usepackage{mathrsfs}
\usepackage{amsthm}
\usepackage{newlfont}

 \newtheorem{thm}{Theorem}
 
 \newtheorem{cor}[thm]{Corollary}

 \newtheorem{lem}[thm]{Lemma}

 \theoremstyle{definition}
 \newtheorem{defn}[thm]{Definition}
 
 \theoremstyle{definition}

 \theoremstyle{remark}
 \newtheorem{rem}[thm]{Remark}
 
 \theoremstyle{definition}
 \newtheorem{example}[thm]{Example}

 \numberwithin{thm}{section}
 \numberwithin{equation}{section}

 \newcommand{\Hom}{\mathrm{Hom}}

 \newcommand{\ord}{\mathrm{ord}}

 \newcommand{\GL}{\mathrm{GL}}
 \newcommand{\PGL}{\mathrm{PGL}}
 
 \newcommand{\SL}{\mathrm{SL}}

\newcommand{\sT}{\mathscr{T}}

\newcommand{\sD}{\mathscr{D}} 
\newcommand{\sB}{\mathscr{B}} 
 
 \newcommand{\fp}{\mathfrak p}
 \newcommand{\fq}{\mathfrak q}
 \newcommand{\fn}{\mathfrak n}
 
 \newcommand{\fm}{\mathfrak m}
 \newcommand{\fd}{\mathfrak d}

 \newcommand{\cO}{\mathcal{O}}

 \newcommand{\cX}{\mathcal{X}}
 \newcommand{\cM}{\mathcal{M}}
 
 \newcommand{\cJ}{\mathcal{J}}
 \newcommand{\cI}{\mathcal{I}}

 \newcommand{\R}{\mathbb{R}}

 \newcommand{\F}{\mathbb{F}}
 \newcommand{\Q}{\mathbb{Q}}
 \newcommand{\Z}{\mathbb{Z}}

 \newcommand{\M}{\mathbb{M}}
 \newcommand{\p}{\mathbb{P}}

 \newcommand{\Nr}{\mathrm{Nr}}

 \newcommand{\To}{\longrightarrow}
 \newcommand{\bs}{\setminus}
 
 \newcommand{\Fi}{F_\infty}

 \newcommand{\G}{\Gamma}
 
 \newcommand{\la}{\lambda}

\begin{document}

\title[On component groups of Jacobians of quaternionic modular curves]{On component groups of Jacobians of quaternionic modular curves}

\author{Mihran Papikian}
\address{Department of Mathematics, Pennsylvania State University, University Park, PA 16802, U.S.A.}
\email{papikian@psu.edu}

\thanks{The author's research was partially supported by grants from the Simons Foundation (245676) and the National Security Agency 
(H98230-15-1-0008).} 

\dedicatory{Dedicated to Ernst-Ulrich Gekeler}

\subjclass[2010]{11G18, 05C25,  11F12} 
\keywords{Graph Laplacian; Component group; Shimura curve.}



\begin{abstract} We use a combinatorial result relating the discriminant of 
the cycle pairing on a weighted finite graph to the eigenvalues of its Laplacian, to deduce 
a formula for the orders of component groups of Jacobians of modular curves 
arising from quaternion algebras over $\F_q(T)$ or $\Q$. Our formula over $\Q$ 
recovers a result of Jordan and Livn\'e. 
\end{abstract}


\maketitle

\section{Introduction}\label{sIntro}

Let $\F_q$ be a finite field with $q$ 
elements, where $q$ is a power of a prime number $p$. Let $A=\F_q[T]$ 
be the ring of polynomials in the indeterminate $T$ with coefficients 
in $\F_q$, and $F=\F_q(T)$ be the fraction field of $A$. 
The degree map $\deg: F\to \Z\cup \{-\infty\}$, which assigns 
to a non-zero polynomial its degree in $T$ and $\deg(0)=-\infty$, is a valuation of $F$. 
The corresponding place of $F$ will be denoted by $\infty$. 
Apart from $\infty$, the places of $F$ are in bijection with the non-zero prime ideals of $A$. Given a place $v$ of $F$, 
we denote by $F_v$ the completion of $F$ at $v$, and by $\F_v$ the residue field of $F_v$. 

Let $\sB$ be a division quaternion algebra over $F$, which is ``indefinite'' in the sense that $\sB\otimes_F \Fi\cong \M_2(\Fi)$.  
Let $\fd$ be the (reduced) discriminant of $\sB$; cf. \cite[p. 58]{Vigneras}.  
This is the product of primes $\fp\lhd A$ where $\sB$ ramifies, i.e., $\sB\otimes_F F_\fp$ 
is a division algebra over $F_\fp$. It is known from the theory of quaternion algebras that $\fd$ 
is a product of an even number of primes, and, up to isomorphism, $\fd$ uniquely determines $\sB$; cf. \cite{Vigneras}.  
(Moreover, any square-free $\fd$ with an even number of prime divisors is the discriminant of some indefinite quaternion algebra over $F$.)

There is a smooth projective curve $X^\fd$ over $F$ associated with the quaternion algebra $\sB$. 
It is the coarse moduli scheme of the so-called $\sD$-elliptic sheaves; roughly, a \textit{$\sD$-elliptic sheaf} 
is a vector bundle equipped with a meromorphic Frobenius and an action of $\sB$; we refer to \cite{LRS} 
for the precise definitions. The curve $X^\fd$ has bad reduction exactly at the primes dividing $\fd$ and $\infty$; cf. \cite{LRS}. 
Let $J^\fd$ be the Jacobian variety of $X^\fd$, and $\cJ^\fd$ be the N\'eron model of $J^\fd$ over $A$.  
Let $\fp$ be a prime dividing $\fd$. The main result of this paper is a formula for the order 
of the component group $\Phi_{J^\fd, \fp}=(\cJ^\fd\otimes \overline{\F}_\fp)/(\cJ^\fd\otimes \overline{\F}_\fp)^0$. 
In order to state this formula, we need to introduce more notation. 

The order of a finite set $S$ will be denoted by $|S|$. We define norm and degree for a non-zero ideal $\fn$ of $A$ 
by $|\fn|:=|A/\fn|$ and $\deg(\fn):=\log_q|\fn|$. 
The prime ideals of $A$ will always be assumed to be non-zero. 
Let $\sB'$ be the ``definite'' quaternion algebra over 
$F$ which is ramified at $\infty$ and at the primes $\fp_1, \dots, \fp_s$ appearing in the prime decomposition of 
$$\fd':=\fd/\fp=\prod_{i=1}^s \fp_i.$$ 
Fix a maximal $A$-order $\cM$ in $\sB'$. 
Let $h(\fd')$ be the class number of $\sB'$, that is, the number of left
$\cM$ ideal classes. It is known that $h(\fd')$ is
finite and independent of the choice of $\cM$. Let $I_1, I_2,\dots, I_{h(\fd')}$ be representatives of 
the left ideal classes of $\cM$, and let $\cM_i$ be the right order of $I_i$, $1\leq i\leq h(\fd')$. 
It is not hard to show that each $\cM_i^\times$ is isomorphic either to $\F_q^\times$ or $\F_{q^2}^\times$;  
thus, $w_i:=|\cM_i^\times/A^\times|$ is either equal to $1$ or $q+1$; see \cite[p. 383]{DvG}. 
It is known that (cf. \cite[(1)]{DvG})  
$$
m(\fd'):=\sum_{i=1}^{h(\fd')} w_i^{-1}=\frac{1}{q^2-1}\prod_{i=1}^s(|\fp_i|-1); 
$$  
$m(\fd')$ is called the \textit{mass} of $\sB'$. If we denote by $h_{q+1}$ the number of $w_i$'s equal to $q+1$, then by 
(1), (4) and (6) in \cite{DvG}:
$$
h_{q+1}=\frac{1}{2}\prod_{i=1}^s\left(1-(-1)^{\deg(\fp_i)}\right). 
$$
This gives a formula for the class number:
$$
h(\fd')=m(\fd')+h_{q+1}\frac{q}{q+1}. 
$$

\begin{defn} Denote $M_{ij}:=I_j^{-1}I_i$. Let 
$m_{ij}$ be a generator of the ideal $\{\Nr(\alpha)\ |\ \alpha\in M_{ij}\}\lhd A$, where $\Nr: \sB'\to F$ is the reduced norm.
Let $\fm\lhd A$ be a non-zero ideal. Let 
$$
b_{ij}(\fm)=\frac{\#\{\beta\in M_{ij}\ |\ (\Nr(\beta)/m_{ij})=\fm\}}{(q-1)w_j}. 
$$
The \textit{$\fm$-th Brandt matrix} of  $\sB'$ is $B(\fm)=\left(b_{ij}(\fm)\right)_{1\leq i,j\leq h(\fd')}\in \M_{h(\fd')}(\Z)$. 
\end{defn}

Let $P(x)$ be the characteristic polynomial of the Brandt matrix $B(\fp)$. Denote by $P'(x)$ 
the derivative of $P(x)$.  
Let 
$$
n(\fd)=\left(1-(-1)^{\deg(\fp)}\right) h_{q+1}
= \frac{1}{2}\prod_{\substack{\fq|\fd \\ \fq\text{ is prime}}}\left(1-(-1)^{\deg(\fq)}\right). 
$$

\begin{thm}\label{corJ^d} We have:
$$
|\Phi_{J^\fd, \fp}| =\frac{\left|P(-|\fp|-1)\cdot P'(|\fp|+1)\right|}{2m(\fd')(q+1)^{n(\fd)}}.
$$
\end{thm}

This formula is the function field analogue of a result of Jordan and Livn\'e \cite{JLComGrp} for Shimura curves 
over $\Q$. In \cite{JLComGrp}, for a prime $p$ 
dividing the discriminant $d$ of an indefinite quaternion algebra, the authors use the 
Cherednik-Drinfeld $p$-adic uniformization of the corresponding Shimura curve $X^d$ to construct a regular model $\cX^d$
of $X^d$ over $\Z_p$. Then, they apply a result of Raynaud (cf. \cite[$\S$9.6]{NM}), reducing  
the problem of computation of the group of connected components $\Phi_{J^d, p}$ of the Jacobian of $X^d$ 
to linear algebra computations with the intersection matrix of the special fibre of $\cX^d$. This approach 
leaves it somewhat mysterious the reason for the appearance of different invariants of the quaternion 
algebra in the formula for $|\Phi_{J^d, p}|$. 

To prove Theorem \ref{corJ^d} we take a different approach, which 
emphasises the combinatorial nature of the formula for $|\Phi_{J^\fd, \fp}| $, although it still relies 
on the function field analogue of the Cherednik-Drinfeld uniformization. In Section \ref{secG&L}, 
we recall a result from \cite{laplacian}, which relates the discriminant of a weighted 
cycle pairing on the first homology group of a graph to the eigenvalues of its Laplacian. 
By a theorem of Grothendieck, given a semistable curve, the discriminant of the cycle pairing on its dual graph 
is the order of the component group of its Jacobian (see Theorem \ref{thmGroth}). 
On the other hand, using the rigid-analytic uniformization of $X^\fd$, one shows that 
the Laplacian of the dual graph of $X^\fd$ over $F_\fp$ is related to a Hecke operator 
with characteristic polynomial $P(x)$. Here we should mention that the dual graph 
of the semistable model of $X^\fd$ that one obtains through the rigid-analytic uniformization 
is naturally weighted and the Hecke operator mentioned above is related to the weighted Laplacian, 
so one has to deal with weighted graphs from the outset. 
The proof of Theorem \ref{corJ^d} then reduces to computing the weights of this dual graph, which is done in Section \ref{ss2Appendix}. 
Our approach applies to Shimura curves over $\Q$ as well. 
In Section \ref{sSC}, we reprove the result of  Jordan and Livn\'e using this method. (Here 
we also use the opportunity to correct a small mistake in the main formula of \cite{JLComGrp}; see Remark \ref{lastrem}.)

\begin{example}
We observe that $h(\fd')=1$ if and only if $\deg(\fd')\leq 2$. Since $\fd'$ is a product of 
an odd number of primes, we conclude that $h(\fd')=1$ is equivalent to $\fd'=\fq$ being a prime 
of degree $\leq 2$. Thus, $\fd=\fp\fq$ and $P(x)=x-(|\fp|+1)$, since $|\fp|+1$ is a root of $P(x)$ for any $\fp$ (cf. \cite[p. 739]{WY}). 
First, assume $\deg(\fq)=1$. In this case, $m(\fd')=(q+1)^{-1}$, and 
$$
n(\fd)=
\begin{cases}
2 & \text{if $\deg(\fp)$ is odd},\\
0 & \text{if $\deg(\fp)$ is even}. 
\end{cases}
$$
Thus, 
$$
|\Phi_{J^\fd, \fp}| =
\begin{cases}
\frac{|\fp|+1}{q+1}, & \text{if $\deg(\fp)$ is odd};\\
(|\fp|+1)(q+1), & \text{if $\deg(\fp)$ is even}. 
\end{cases}
$$
Now assume $\deg(\fq)=2$. In this case, $m(\fd')=1$, $n(\fd)=0$, and 
$$
|\Phi_{J^\fd, \fp}| =|\fp|+1. 
$$
\end{example}

\begin{example} 
Some calculations of Brandt matrices over $\F_q(T)$ were carried out by Schweizer in \cite{SchweizerJNT}. 
For example, he computed that for $q=2$, $\fp=T$, and $\fd'=T^5+T^2+1$ we have 
\begin{align*}
P(x)=(x - &3) (x^2+x-1) \\ 
&\times (x^8+x^7-11x^6-8x^5+38x^4+16x^3-44x^2-4x+4). 
\end{align*}
Therefore, for $\fd=T(T^5+T^2+1)$ and $\fp=T$, we have $h(\fd')=11$, $m(\fd')=31/3$, $n(\fd)=2$ and 
$$
|\Phi_{J^\fd, T}|=\frac{|P(-3)P'(3)|}{2\frac{31}{3}3^{2}}= 5^2\cdot 11\cdot 61\cdot 113. 
$$
Nowadays, Brandt matrix calculations over $\F_q(T)$ are implemented in \texttt{Magma}, which  
makes the calculation of $|\Phi_{J^\fd, \fp}|$ 
fairly routine. Some of the values of $|\Phi_{J^\fd, \fp}|$ are listed in Tables \ref{tablePhip1} and \ref{tablePhip2}. 
\begin{table}
\begin{tabular}{|c|c|}
\hline
$\fp$ & $|\Phi_{J^\fd, \fp}|$\\
\hline\hline
$T-3$ & $2^5$ \\
\hline
$T^2-3$ & $2^{13}\cdot 3^6\cdot 5\cdot 7\cdot 11\cdot 13\cdot 17\cdot 37\cdot 59$ \\
\hline
$T^3+T+1$ & $2^{16}\cdot 3^9\cdot 5\cdot 7\cdot 11\cdot 13\cdot 23\cdot 181$ \\
\hline
\end{tabular}
\caption{$q=5$ and $\fd=T(T-1)(T-2)\fp$}\label{tablePhip1}
\end{table}

\begin{table}
\begin{tabular}{|c|c|}
\hline
$\fp$ & $|\Phi_{J^\fd, \fp}|$\\
\hline\hline
$T$ & $2^{10}\cdot 19^2\cdot 43^3$ \\
\hline
$T-1$ & $2^{12}\cdot 3\cdot 7\cdot 179\cdot 1871$ \\
\hline
$T-3$ & $2^{19}\cdot 5\cdot 103^2$ \\
\hline
$T^2-3$ & $2^{23}\cdot 5^2\cdot 7\cdot 29\cdot 6525499\cdot 13235507$ \\
\hline
$T^3+2$ & $2^{22}\cdot 3^3\cdot 5^6\cdot 7\cdot 13^4\cdot 19^5\cdot 43^7$ \\
\hline
\end{tabular}
\caption{$q=7$ and $\fd=(T^3-2)\fp$}\label{tablePhip2}
\end{table}
\end{example}

Let $X_0(\fd)$ be the Drinfeld modular curve over $F$ classifying rank-$2$ Drinfeld $A$-modules 
with $\G_0(\fd)$-level structures; cf. \cite{GR}. Let $J_0(\fd)$ be the Jacobian variety of $X_0(\fd)$.  
The Jacquet-Langlands correspondence and Zarhin's isogeny theorem over function fields imply 
that there is a surjective homomorphism $J_0(\fd)\to J^\fd$ defined 
over $F$; cf. \cite[Thm. 7.1]{PapikianJL}. Thus, there are functorial homomorphisms $\psi_\fp: \Phi_{J_0(\fd), \fp}\to \Phi_{J^\fd, \fp}$ 
for primes $\fp$ dividing $\fd$, 
where $\Phi_{J_0(\fd), \fp}$ denotes the group of connected components of the N\'eron model of $J_0(\fd)$ at $\fp$. 
The group $\Phi_{J_0(\fd), \fp}$ is well-understood: its complete description as an abelian group 
can be found in \cite[Thm. 5.3]{PW}. In particular, $\Phi_{J_0(\fd), \fp}$ has no elements of $p$-power order and the isomorphism class 
of $\Phi_{J_0(\fd), \fp}$ depends only on the degrees of the 
primes dividing $\fd$ and the number of these primes. We observe from Tables \ref{tablePhip1} and \ref{tablePhip2} that 
these properties do not transfer to $\Phi_{J^\fd, \fp}$, in the sense that $|\Phi_{J^\fd, \fp}|$ can be divisible by the characteristic of the field and 
$|\Phi_{J^\fd, \fp}|$ depends not just on the degree of $\fp$ but on the actual ideal $\fp$. 
Another difference between the groups $\Phi_{J_0(\fd), \fp}$ and $\Phi_{J^\fd, \fp}$ is that the latter is generally a much larger group. 
Indeed, the order of the group $\Phi_{J_0(\fd), \fp}$ grows linearly with $|\fd|$; see \cite[Thm. 5.3]{PW}. On the other hand,  
one can use Theorem \ref{corJ^d} and the arguments in \cite{laplacian} to prove the following estimates: 

\begin{thm} For two positive real valued functions $f(x)$ and $g(x)$ defined on an infinite subset of ideals of $A$, we write 
$f(x)\sim g(x)$ when $\lim_{|x|\to \infty}f(x)/g(x)=1$. We have:  
\begin{enumerate}
\item If $\fd'$ is fixed and $\fp$ varies, then 
$$
\log_q |\Phi_{J^\fd, \fp}|\sim (2h(\fd')-1)\deg(\fp). 
$$
\item If $\fp$ is fixed and $\fd'$ varies, then 
$$
\ln |\Phi_{J^\fd, \fp}|\sim 2h(\fd')\cdot c(\fp), 
$$
where $c(\fp)$ is an explicit constant depending only on $\fp$, which can be estimated as 
$c(\fp)=\ln\left(|\fp|+\frac{1}{2}\right)+O\left(|\fp|^{-2}\deg(\fp)\right)$. 
\end{enumerate}
\end{thm}
\begin{proof} It is known from the theory of automorphic forms that one of the roots of $P(x)$ is $\la_1=|\fp|+1$ 
and all other roots satisfy $|\la_i|\leq 2\sqrt{|\fp|}$, $2\leq i\leq h(\fd')$. 
More precisely, the Brandt matrices form an algebra which is isomorphic to 
the algebra of Hecke operators acting on the space spanned by one specific Eisenstein series 
and the ``$\fp$-new'' part of Drinfeld automorphic cusp forms on the congruence subgroup $\G_0(\fd)$ of $\GL_2(A)$; cf. \cite[$\S$2]{WY}. 
Using this, one obtains the bound $|\la_i|\leq 2\sqrt{|\fp|}$ from the Ramanujan-Petersson conjecture 
over function fields. Part (1) then easily follows from Theorem \ref{corJ^d}. 

If $\fp$ is fixed and $\deg(\fd')\to \infty$, then the eigenvalues $\la_2, \dots, \la_{h(\fd')}$ 
become equidistributed in the interval $\left[-2\sqrt{|\fp|}, 2\sqrt{|\fp|}\right]$ with respect to 
a certain Sato-Tate measure. 
This follows from the equidistribution result of Nagoshi \cite{Nagoshi} for the eigenvalues 
of a fixed Hecke operator acting on the spaces of Drinfeld cusp forms as the level varies, and the observation of Serre \cite[Thm. 2]{SerreJAMS} 
that the same equidistribution persists if one restricts the Hecke operator to the new part of the space of cusp forms. 
Having this fact, one can argue as in \cite[$\S$4.3]{laplacian} to obtain part (2). 
\end{proof}

\subsection*{Acknowledgements} This work was carried out while I 
was visiting the Taida Institute for Mathematical Sciences in Taipei and the Max Planck Institute for Mathematics in Bonn. 
 I thank these institutes for their hospitality and excellent working conditions. 
 I am grateful to Fu-Tsun Wei for very useful discussions related to the topic of this paper.  


\section{Graphs and Laplacians}\label{secG&L} 

\subsection{Graphs}\label{ssG} A \textit{graph} consists of a set of vertices $V(G)$, a set of oriented edges $E(G)$ 
and two maps 
$$
E(G)\to V(G)\times V(G), \quad 
e\mapsto (o(e), t(e))
$$
and 
$$
E(G)\to E(G), \quad e\mapsto \bar{e}
$$ 
such that $\bar{\bar{e}}=e$ and $t(\bar{e})=o(e)$; cf. \cite[p. 13]{SerreT}.  
For $e\in E(G)$, the edge $\bar{e}$  is called the \textit{inverse} of $e$, the vertex $o(e)$ (resp. $t(e)$) is called   
the \textit{origin} (resp. \textit{terminus}) of $e$. The vertices $o(e), t(e)$ are called the \textit{extremities} 
of $e$. We say that two vertices are \textit{adjacent} if they are the extremities of some edge. We say 
that $G$ is \textit{finite}, if it has finitely many vertices and edges. 
A \textit{weighted graph} is a graph $G$ equipped with two maps 
$w: E(G)\to \Z_{>0}$ and $w: V(G)\to \Z_{>0}$, called weight functions,  such that $w(e)=w(\bar{e})$ for all $e\in E(G)$. 
We say that a weighed graph $G$ 
is \textit{$N$-regular} if there is a positive integer $N\geq 1$ such that for any vertex $v\in V(G)$ we have  
$$
\sum_{\substack{e\in E(G)\\ t(e)=v}} \frac{w(v)}{w(e)}=N.  
$$

Note that two vertices might be the extremities of multiple edges
\begin{tikzpicture}[scale=.5]
\filldraw [black] (0,0) circle (2pt);
\filldraw [black] (1,0) circle (2pt);
\draw  (0,0) arc (180:360:.5cm);
\draw  (1,0) arc (0:180:.5cm);
\end{tikzpicture}
, i.e., there might be $e\neq e'$ 
with $o(e)=o(e')$ and $t(e)=t(e')$.  
Also, $G$ might have loops
\begin{tikzpicture}[scale=.5]
\filldraw [black] (0,0) circle (2pt);
\draw  (0,0) arc (180:360:.5cm);
\draw  (1,0) arc (0:180:.5cm);
\end{tikzpicture}
, i.e., edges $e\in E(G)$ with $e\neq \bar{e}$ and $t(e)= o(e)$.  
Finally, as in \cite{Kurihara}, we do not exclude the case $e=\bar{e}$: 
  \begin{tikzpicture}[scale=.5]
\filldraw [black] (0,0) circle (2pt);
\draw (0,0) -- (1,0);
\end{tikzpicture}

We say that a group $\G$ acts on the weighted graph $G$ if $\G$ acts on the sets $V(G)$ and $E(G)$ 
so that for any $\gamma\in \G$, $v\in V(G)$, $e\in E(G)$ 
we have $\gamma (o(e))=o(\gamma e)$, $\overline{\gamma e}=\gamma \bar{e}$, $w(\gamma e)=w(e)$, $w(\gamma v)=w(v)$.  
For a group $\G$ acting on $G$, we have a quotient graph $\G\bs G$ and a natural mapping $j: G\to \G\bs G$ such that $V(\G\bs G)=\G\bs V(G)$ 
and  $E(\G\bs G)=\G\bs E(G)$. Furthermore, if $G$ is a weighted graph and the stabilizers 
$$
\G_v:=\{\gamma \in \G\ |\ \gamma v=v\}\quad \text{and}\quad \G_e:=\{\gamma \in \G\ |\ \gamma e=e\}
$$
are finite for every $v\in V(G)$, $e\in E(G)$, then we make $\G\bs G$ into a weighted graph by putting $w(j(v))=|\G_v|\cdot w(v)$ 
and $w(j(e))=|\G_e|\cdot w(e)$.  

Let $K$ be a non-archimedean local field with finite residue field $k$. Let $\sT$ be the 
Bruhat-Tits tree associated with $\SL_2(K)$; cf. \cite{SerreT}. We make $\sT$ into a weighed graph by assigning the 
weight $1$ to all its vertices and edges. Then $\sT$ is a $(|k|+1)$-regular tree. 
The group $\PGL_2(K)$ naturally acts on $\sT$. Let $\G$ 
be a discrete subgroup of $\PGL_2(K)$ with compact quotient. Then the stabilizers in $\G$ of the vertices and edges of $\sT$ 
are finite, and, since $\G$ is cocompact in $\PGL_2(K)$, the quotient graph $\G\bs \sT$ is finite; cf. \cite[Lem. 5.1]{PapikianCrelle}.  
We make $\G\bs \sT$ into a weighed graph as in the previous paragraph.

\subsection{Laplacians}\label{ssL}
We assume from now on that $G$ is finite and connected. 
Let $C_0(G, \Q)$ be the vector space over $\Q$ with basis $V(G)$. The (weighted) \textit{Laplacian} is the linear 
transformation $\Delta: C_0(G, \Q)\to C_0(G, \Q)$ defined by 
$$
\Delta(v)=\sum_{t(e)=v}\frac{w(v)}{w(e)}(v-o(e)). 
$$
It is well-known (and easy to prove) that the eigenvalues of $\Delta$ are non-negative real numbers, 
and $0$ is an eigenvalue of $\Delta$ with multiplicity one; cf. \cite[$\S$2.1]{laplacian}. The weighted 
\textit{adjacency operator} is the linear 
transformation $\delta: C_0(G, \Q)\to C_0(G, \Q)$ defined by 
$$
\delta(v)=\sum_{t(e)=v}\frac{w(v)}{w(e)}o(e). 
$$

The pairing  
$E(G) \times E(G) \to \Z$: 
$$
( e, e') = 
\begin{cases} 
w(e) & \text{if $e'=e$},\\
-w(e) & \text{if $e'=\bar{e}$},\\
0 & \text{otherwise}, 
\end{cases}
$$
extends to a symmetric, bilinear, positive-definite pairing on the first simplicial 
homology group $H_1(G, \Z)$, which is just a weighted version of the usual cycle pairing.
Let $D(G)$ be the order of the cokernel of the map 
\begin{align*}
H_1(G, \Z) &\To \Hom(H_1(G, \Z), \Z) \\
\varphi & \mapsto ( \varphi, \ast ). 
\end{align*} 

For a graph $G$, let $E(G)^\ast$ be a subset of $E(G)$ with the following two properties: 
\begin{itemize}
\item[$(i)$] if $e=\bar{e}$ then $e\not \in E(G)^\ast$; 
\item[$(ii)$] if $e\neq \bar{e}$ then either $e$ or $\bar{e}$, but not both, is in $E(G)^\ast$. 
\end{itemize}
We define the \textit{mass} of $G$ as 
$$
m(G):=\sum_{v\in V(G)} w(v)^{-1}. 
$$
\begin{thm}\label{thmMyDelta}
Assume $G$ is finite with $n$ vertices. Let 
$$
0=\la_1<\la_2\leq \cdots \leq \la_n
$$
be the eigenvalues of $\Delta$. Then 
$$
D(G) = m(G)^{-1} \frac{\prod_{e\in E(G)^\ast} w(e)}{\prod_{v\in V(G)} w(v)} \prod_{i=2}^n \la_i. 
$$
\end{thm}
\begin{proof}
See Theorem 3.1 in \cite{laplacian}. 
\end{proof}

The arithmetic significance of this theorem comes from the following. Let $\pi: X\to S$ be a 
semi-stable curve of genus $g\geq 1$ over a Dedekind scheme $S$ of dimension $1$. Recall that this means that $\pi$ 
is a proper and flat morphism whose fibres $X_{\bar{s}}$ over the geometric points $\bar{s}$ of $S$ are 
reduced, connected curves of arithmetic genus $g$, and have only ordinary double points as singularities; cf. \cite[Def. 10.3.14]{Liu}.  
We assume that the generic fibre $X_\eta$ of $X$ is a smooth, projective, geometrically connected curve. 
Let $s\in S$ be a closed point, and $x\in X_s$ be a singular point. There exists 
a scheme $S'$, \'etale  over $S$, such that any point $x'\in X':=X\times_S S'$ 
lying above $x$, belonging to a fiber $X_{s'}'$, is a split ordinary double point, and 
$$
\widehat{\cO}_{X', x'}\cong \widehat{\cO}_{S', s'}[\![u, v]\!]/(uv-c)
$$
for some $c\in \cO_{S', s'}$. Moreover, the valuation $w_x$ of $c$ for the normalized valuation of $\cO_{S', s'}$ 
is independent of the choice of $S', s'$, and of $x'$. For the proof of these facts we refer to \cite[Cor. 10.3.22]{Liu}. 

One can associate a graph $G_{X_s}$ to $X_s$, the so-called \textit{dual graph} (cf. \cite[p. 511]{Liu}): 
Let $k_s$ be the residue field at $s$. 
The vertices of $G_{X_s}$ are the irreducible components of $X_s\times_{k_s} \overline{k_s}$, and each 
ordinary double point $x\in X_s$ defines an edge $e_x$ whose extremities correspond to the irreducible components containing $x$ 
(the two orientations of $e_x$ correspond to a choice of one of the two branches passing through $x$ as the 
origin of $e_x$). We assign the weight $w(e_x)=w_x$.  

\begin{thm}\label{thmGroth} Let $J$ be the Jacobian variety of $X_\eta$. Let $\cJ$ be the N\'eron model of $J$ 
over $S$. 
Let $\Phi_{J, s}:=\cJ_{\bar{s}}/\cJ_{\bar{s}}^0$ be the group of connected components of $\cJ$ at $s\in S$. 
Then $|\Phi_{J,s}|=D(G_{X_s})$. 
\end{thm}
\begin{proof}
This follows from 11.5 and 12.10 in \cite{SGA7}. 
\end{proof} 

Thus, Theorems \ref{thmMyDelta} and \ref{thmGroth} relate the order $|\Phi_{J,s}|$ to the eigenvalues 
of the weighted Laplacian of $G_{X_s}$. 
\begin{rem}
Note that we have assigned weights only to the edges of $G_{X_s}$, and 
in fact, $|\Phi_{J,s}|$ depends only on the weights of the edges of $G_{X_s}$. To apply 
the formula in Theorem \ref{thmMyDelta}, one can assign arbitrary weights to the vertices of $G_{X_s}$, e.g., all $1$. 
As we will see, in the situations where $X$ arises from quaternion algebras, there will be a natural choice 
for the weights of vertices.  
\end{rem}

\subsection{Involutions}\label{ssID} 
Let $G$ be a weighed $N$-regular finite connected graph. In addition, assume that $G$ is bipartite, that is, 
$V(G)$ is a disjoint union $V(G)=O\sqcup I$ such that 
every edge $e\in E(G)$ has one of its extremities in $O$ and the other in $I$; in particular, $G$ 
has no edges with $t(e)=o(e)$. 

Let $\tau$ be a non-trivial involution acting on $G$ (as a weighted graph) which interchanges $O$ and $I$.
Note that $\tau$ 
does not fix any vertices or edges, although it is possible that $\tau(e)=\bar{e}$. 
Let $G':=G/\langle \tau\rangle$ be the quotient of $G$ under the action of $\tau$. 
This graph $G'$ is a weighted graph in which $w(v')=w(v)$ and $w(e')=w(e)$, 
where $v\in V(G)$ is a preimage of $v'\in V(G')$ and $e\in E(G)$ is a preimage of $e'\in V(G')$. 
Note that $G'$ might have edges $e'$ with $\bar{e'}=e'$;  
this happens when $\tau(e)=\bar{e}$ for the preimage $e\in E(G)$ of $e'$. Also, $G'$ 
might contain loops; this happens if there are two edges $e_1\neq e_2$ in $G$ 
such that $o(e_1)=o(e_2)$, $t(e_1)=t(e_2)$, and $\tau(e_1)=\bar{e}_2$.  Finally, note that $G'$ 
is again $N$-regular.

Let $\Delta: C_0(G, \Q)\to C_0(G, \Q)$ be the weighted Laplacian of $G$: 
$$
\Delta(v)=Nv-\delta(v),
$$
where $\delta$ is the weighted adjacency operator of $G$. Let $\Delta'$ be the weighted Laplacian of $G'$
and $\delta'$ its weighted adjacency operator. 

\begin{lem}\label{lempmeigenvalues}
Let $h$ be the number of vertices of $G'$. If $\la_1, \dots, \la_h$ are the eigenvalues of $\delta'$, then 
$\pm \la_1, \dots, \pm \la_h$ are the eigenvalues of $\delta$. 
\end{lem}
\begin{proof} It is clear that $G$ has $2h$ vertices. 
It is also clear that $\delta$ and $\tau$ commute as linear operators on $C_0(G, \R)$. 
Let $v_1, \dots, v_h$ be the vertices in $O$, so 
that $\tau(v_1), \dots, \tau(v_h)$ are the vertices in $I$. Let $v_1', \dots, v_h'\in V(G')$ 
be the images of $v_1, \dots, v_h$; we take these vertices as a basis of $C_0(G', \R)$. Let $M$ 
be the matrix of $\delta'$ with respect to this basis. From what was said, it is easy to see that $\delta$ 
is represented by the $2h\times 2h$ matrix $\begin{pmatrix} 0 & M\\ M & 0\end{pmatrix}$. 
Since $$\begin{pmatrix} I_h & I_h\\ I_h & -I_h\end{pmatrix}\begin{pmatrix} 0 & M\\ M & 0\end{pmatrix}
\begin{pmatrix} I_h & I_h\\ I_h & -I_h\end{pmatrix}^{-1}=\begin{pmatrix} M & 0\\ 0 & -M\end{pmatrix},$$
the claim follows.  
\end{proof}

The largest eigenvalue of $\delta'$ is $N$, which occurs with multiplicity $1$. 

\begin{cor}\label{corPpolynomial} Let $P(x)$ be the characteristic polynomial of $\delta'$. Let $P'(x)$ be the derivative of $P(x)$. 
Enumerate the eigenvalues of $\delta'$ so that 
$
N=\la_1> \la_2\geq \cdots\geq \la_h. 
$
The eigenvalues of $\Delta$ are $N\pm \la_i$, $1\leq i\leq h$, and the product of non-zero eigenvalues of $\Delta$ is 
$$
2N\prod_{i=2}^h (N+\la_i)(N-\la_i)=|P(-N)P'(N)|. 
$$
\end{cor}
\begin{proof}
The first claim follows from Lemma \ref{lempmeigenvalues}. The last equality follows from 
$$
P'(N) =\prod_{i=2}^h(N-\la_i), \qquad  
P(-N) =(-1)^h 2N\prod_{i=2}^h(N+\la_i).
$$
\end{proof}


\section{Graphs arising from quaternion algebras}\label{ss2Appendix}

The set-up and notation in this section will be the same as in the introduction. 
In particular, $\sB$ is an indefinite quaternion algebra over $F$ of discriminant $\fd$.   
We fix a prime $\fp | \fd$, and denote by $\sB'$ the definite quaternion algebra 
of discriminant $\fd':=\fd/\fp$. 
Let $K:=F_\fp$, and 
\begin{align*}
\G_0 & :=\left(\cM\otimes A[\fp^{-1}]\right)^\times/ A[\fp^{-1}]^\times, \\
\G_+ & :=\left\{m\in \left(\cM\otimes A[\fp^{-1}]\right)^\times \ \big|\ \ord_\fp\Nr(m)\in 2\Z \right\}\big/A[\fp^{-1}]^\times.  
\end{align*}
Note that $\G_+$ is a normal subgroup of $\G_0$ of index $2$. 
By fixing an isomorphism $\sB'\otimes_F K\cong \M_2(K)$, 
one can consider $\G_0$ and $\G_+$ as discrete cocompact subgroups of $\PGL_2(K)$. 
Let $\sT$ be the Bruhat-Tits tree associated with $\SL_2(K)$. 
As we have discussed in $\S$\ref{ssG}, the quotient graphs 
$G_0:=\G_0\bs \sT$ and $G_+:=\G_+\bs \sT$ are finite, weighted, $(|\fp|+1)$-regular graphs. 

Let $V(\sT)=V(\sT)_1\sqcup V(\sT)_2$ be the disjoint union such that, for $v\in V(\sT)_i$ 
and $u\in V(\sT)_j$, the combinatorial distance between $v$ and $u$ is even if and only if $i=j$. 
Using the Corollary on page 75 in \cite{SerreT}, it is easy to see that $\gamma V(\sT)_i=V(\sT)_i$ ($i=1,2$) 
for $\gamma\in \G_+$, and $\gamma V(\sT)_1=V(\sT)_2$ and $\gamma V(\sT)_2=V(\sT)_1$ 
for $\gamma\in \G_0-\G_+$. Thus, $G_+$ is bipartite. 
Fix some $\tau\in \G_0 - \G_+$. It induces an involution of $G_+$ which does not depend on the choice of $\tau$. 
By abuse of notation, we will denote this involution by $\tau$. 
The triple $\{G_+$, $\tau$, $G_0=G_+/\langle \tau \rangle\}$ fits into the setting of $\S$\ref{ssID}. 

As is explained in \cite[p. 291]{Kurihara}, using Eichler's approximation theorem, one can 
identify $V(G_0)=\{v_1, \dots, v_{h(\fd')}\}$ with the set $\{\cM_1, \dots, \cM_{h(\fd')}\}$ so that 
$(\G_0)_{v_i'}\cong \cM_i^\times/A^\times$, where $v_i'$ is a preimage of $v_i$ in $\sT$. 
In particular, the cardinality of $V(G_0)$ is the class number of $\sB'$ and 
$w(v_i)=w_i$, $1\leq i\leq h(\fd')$. 
Moreover, by \cite[p. 294]{Kurihara} (see also \cite[Prop. 4.4]{Gross}), 
the weighted adjacency operator of $G_0$ is given by the transpose of the $\fp$-th Brandt matrix $B(\fp)$ if we fix $v_1, \dots, v_{h(\fd')}$ 
as a basis of $C_0(G_0, \R)$. 
(In \cite{Kurihara}, Kurihara considers only quaternion algebras over $\Q$, but his arguments 
apply also, essentially verbatim, to quaternion algebras over $F$.)

\begin{lem}\label{lemKurihara}
Let $v\in V(G_0)$ be a vertex of weight $q+1$. For an edge $e\in E(G_0)$ with $t(e)=v$ we have either $w(e)=1$ 
or $w(e)=q+1$. Moreover, 
$$
\#\left\{e\in E(G_0)\ |\ t(e)=v,\ w(e)=q+1\right\}=1+(-1)^{\deg(\fp)}. 
$$
\end{lem}
\begin{proof} Let $e'\in E(\sT)$ be a preimage of $e\in E(G_0)$. Similar to the case of vertices, 
one can show that the stabilizer $(\G_0)_{e'}$ is isomorphic to $\cI^\times/A^\times$ for a certain Eichler $A$-order $\cI$ in $\sB'$. 
On the other hand, by \cite[p. 383]{DvG}, $\cI^\times$ is isomorphic either to $\F_q^\times$ or $\F_{q^2}^\times$. 
This implies the first claim. 

The second claim can also be deduced from the results in \cite{DvG}, but a more direct 
argument is the following:
Let $v'\in V(\sT)$ be a preimage of $v$. Through the identification discussed in the introduction, let $\cM'$ be the maximal order corresponding to $v'$. 
Then ${\cM'}^\times\cong \F_{q^2}^\times$. Let $\gamma$ be a generator of this group. We need to count the number of 
edges $e'\in E(\sT)$ such that $t(e')=v'$ and $\gamma e'=e'$. The set $$\mathrm{Star}(v')=\{e'\in E(\sT)\ |\ t(e')=v'\}$$ 
is in natural bijection with $\p^1(\F_\fp)$, and through this bijection $\gamma$ acts on $\mathrm{Star}(v')$ in the same way 
as a certain matrix in $\GL_2(\F_q)$ acts on $\p^1(\F_\fp)$; cf. \cite[p. 292]{Kurihara}. 
The characteristic polynomial of $\gamma$, as a matrix in $\GL_2(\F_q)$, is an irreducible quadratic polynomial over $\F_q$. 
It is easy to check that such a matrix has two fixed points in $\p^1(\overline{\F}_q)$ and these points lie 
in $\p^1(\F_{q^2})-\p^1(\F_q)$. Therefore, $\gamma$ acting on $\p^1(\F_\fp)$ has no fixed points if $\deg(\fp)$ is odd, 
and has exactly two fixed points if $\deg(\fp)$ is even. 
\end{proof}

\begin{thm}\label{thmG+} The graph $G_+$ has the following properties:
\begin{enumerate}
\item $G_+$ is a finite, weighted, $(|\fp|+1)$-regular graph. 
\item $|V(G_+)|=2h(\fd')$.
\item For $v\in V(G_+)$, we have $w(v)=1$ or $q+1$. The number of vertices of weight $q+1$ is $2h_{q+1}$.  
\item $\sum_{v\in V(G_+)}w(v)^{-1}=2m(\fd')$. 
\item For $e\in E(G_+)$, we have $w(e)=1$ or $q+1$. Modulo the 
orientation, the number of edges of $G_+$ of weight $q+1$ is $h_{q+1}\left(1+(-1)^{\deg(\fp)}\right)$.  
\item 
$$
D(G_+) =\frac{\left|P(-|\fp|-1)\cdot P'(|\fp|+1)\right|}{2m(\fd')(q+1)^{n(\fd)}}.\\
$$
\end{enumerate}
\end{thm}
\begin{proof}
We have $|V(G_+)|=2|V(G_0)|=2h(\fd')$, the weight of $v\in V(G_+)$ 
is the weight of its image in $G_0$, and every $v\in V(G_0)$ has exactly two preimages in $G_+$. This proves (1)-(4). 
Similarly, the weight of $e\in E(G_+)$ is equal to the weight of its image in $G_0$. Therefore, by Lemma \ref{lemKurihara}, for 
$e\in E(G_+)$ we have $w(e)=1$ or $q+1$. We can count the number of edges in $G_0$ of weight $q+1$ using the same lemma. 
More precisely, for each of the $h_{q+1}$ vertices of $G_0$ with weight $q+1$ we get $\left(1+(-1)^{\deg(\fp)}\right)$ 
edges $e$ of weight $q+1$ terminating at the given vertex. This way we count each edge with non-trivial weight twice if $e\neq \bar{e}$ (once from $t(e)$ 
and once from $t(\bar{e})$), and count the edge only once if $e=\bar{e}$. On the other hand, an edge of $G_0$ has two preimages in $G_+$ 
if $e\neq \bar{e}$, and one preimage if $e=\bar{e}$. Thus, modulo the 
orientation, the number of edges of $G_+$ of weight $q+1$ is $h_{q+1}\left(1+(-1)^{\deg(\fp)}\right)$.
We proved that 
$$
\frac{\prod_{e\in E(G_+)^\ast}w(e)}{\prod_{v\in V(G_+)}w(v) } = \frac{(q+1)^{h_{q+1}\left(1+(-1)^{\deg(\fp)}\right)}}{(q+1)^{2h_{q+1}}} 
= \frac{1}{(q+1)^{n(\fd)}}. 
$$
Now the formula in part (6) follows from Theorem \ref{thmMyDelta} and Corollary \ref{corPpolynomial}. 
\end{proof}

As follows from the analogue of the Cherednik-Drinfeld uniformization
for $X^\fd_{K}$, proven in this context by Hausberger
\cite{Hausberger}, $X^\fd_{K}$ is a twisted Mumford curve. More precisely, if we denote by
$\cO_\fp^{(2)}$ the ring of integers of the quadratic unramified extension of $K$ and
denote by $\F_\fp^{(2)}$ the residue field of $\cO_\fp^{(2)}$, then
$X^\fd_{K}$ has a semi-stable model $X^\fd_{\cO_\fp^{(2)}}$ over
$\cO_\fp^{(2)}$ such that the dual graph of $X^\fd_{\F_\fp^{(2)}}$
is the weighted graph $G_+$. Therefore, using Theorem \ref{thmGroth} and Theorem \ref{thmG+}, we obtain Theorem \ref{corJ^d}.


\section{Shimura curves over $\Q$}\label{sSC} Let $\sB$ be an indefinite quaternion algebra over $\Q$
of discriminant $d$. Let $X^d$ be the corresponding Shimura curve over $\Q$ 
parametrizing abelian surfaces equipped with an action of $\sB$. The Jacobian variety $J^d$ of $X^d$ 
has bad reduction at the primes dividing $d$. 

Fix a prime $p$ dividing $d$. 
Let $\sB'$ be the rational definite quaternion algebra of discriminant $d'=d/p$. Fix a maximal 
$\Z$-order $\cM$ in $\sB'$. Let $h(d')$ be the class number of $\sB'$. Let $I_1, I_2, \dots, I_{h(d')}$ 
be representatives of the left ideal classes of $\cM$, and let $\cM_i$ be the right order of $I_i$, $1\leq i\leq {h(d')}$. 
Denote $w_i=\cM_i^\times/\Z^\times$. We have Eichler's formula for the mass of $\sB'$: 
$$
m(d'):=\sum_{i=1}^{h(d')} w_i^{-1}=\frac{1}{12}\prod_{\substack{\ell | d'\\ \ell \text{ prime}}}(\ell-1). 
$$  
 Let $B(p)\in \M_h(\Z)$ be the $p$-th Brandt matrix for $\sB'$, 
and $P(x)$ be the characteristic polynomial of $B(p)$.  
Write $(p+1)=\la_1>\la_2\geq  \dots \geq \la_{h(d')}$ for the eigenvalues of $B(p)$. 

Let 
$$
\G_+=\left\{\gamma \in \left(\cM\otimes \Z[p^{-1}]\right)^\times\ \bigg|\ \ord_p\Nr(\gamma)\in 2\Z \right\}\bigg/ \Z[p^{-1}]^\times. 
$$
The group $\G_+$ acts on the Bruhat-Tits tree $\sT$ of $\SL_2(\Q_p)$. The quotient $G_+:=\G_+\bs \sT_p$ 
is a finite weighted $(p+1)$-regular bipartite graph. 

The Cherednik-Drinfeld uniformization theorem for $X^d_{\Q_p}$ implies that this curve has a semi-stable model over $\Z_p$ 
whose dual graph, as a weighted graph, is $G_+$; see \cite{BC} or \cite{Kurihara}. Therefore, the order 
of the component group $\Phi_{J^d, p}$ of the N\'eron model of $J^d$ at $p$ is $|\Phi_{J^d, p}|=D(G_+)$.   
Now the same argument as in Section \ref{ss2Appendix}, with the same references to \cite{Kurihara}, produces the following 
formula:
$$
|\Phi_{J^d, p}|=\frac{|P(-p-1)P'(p+1)|}{2m(d')}\frac{\prod_{e\in E(G_+)^\ast}w(e)}{\prod_{v\in V(G_+)}w(v)}. 
$$
The only small difference from the function field case is that now the weights $w(v)$, $w(e)$ for $v\in V(G_+)$ and $e\in E(G_+)$ are more varied, 
in the sense that $w(v)$ can be $1,2,3,6,12$, and $w(e)$ can be $1, 2,3$. We consider three separate cases:

Case 1: $d'=2$ (equiv. $d=2p$). In this case, $h(d')=1$, $m(d')=1/12$, $P(x)=x-(p+1)$, and $G_+$ has two vertices of weight $12$.  
By Proposition 4.2 in \cite{Kurihara}, the number of edges of weight $2$ is $\frac{1}{2}\left(1+\left(\frac{-4}{p}\right)\right)$, 
and the number of edges of weight $3$ is $\left(1+\left(\frac{-3}{p}\right)\right)$. 
(Here $\left(\frac{\ast}{p}\right)$ is the Legendre symbol.)
Therefore
\begin{equation}\label{eqJL1}
|\Phi_{J^d, p}|=(p+1)2^{\frac{1}{2}\left(-3+\left(\frac{-4}{p}\right)\right)}3^{\left(\frac{-3}{p}\right)}. 
\end{equation}

Case 2: $d'=3$ (equiv. $d=3p$). In this case, $h(d')=1$, $m(d')=1/6$, $P(x)=x-(p+1)$, and $G_+$ has two vertices of weight $6$.
By Proposition 4.2 in \cite{Kurihara}, the number of edges of weight $2$ is $\left(1+\left(\frac{-4}{p}\right)\right)$, 
and the number of edges of weight $3$ is $\frac{1}{2}\left(1+\left(\frac{-3}{p}\right)\right)$. Therefore
\begin{equation}\label{eqJL2}
|\Phi_{J^d, p}|=(p+1)2^{\left(\frac{-4}{p}\right)}3^{\frac{1}{2}\left(-1+\left(\frac{-3}{p}\right)\right)}. 
\end{equation}

Case 3: $d'\geq 5$. In this case, the vertices of $G_+$ have weights $1, 2$, or $3$. Moreover, by \cite[p. 291]{Kurihara}, the number of 
vertices of weight $2$ (resp. $3$) is $2h_2$ (resp. $2h_3$), where 
$$
h_2=\frac{1}{2}\prod_{\ell |d'}\left(1-\left(\frac{-4}{\ell}\right)\right),
\qquad 
h_3=\frac{1}{2}\prod_{\ell |d'}\left(1-\left(\frac{-3}{\ell}\right)\right). 
$$
By \cite[Prop. 4.2]{Kurihara} and the argument in Section \ref{ss2Appendix}, the number of edges (modulo the orientation) of weight $2$ (resp. $3$) 
in $G_+$ is $h_2\left(1+\left(\frac{-4}{p}\right)\right)$  (resp. $h_3\left(1+\left(\frac{-3}{p}\right)\right)$). 
Therefore, 
$$
\frac{\prod_{e\in E(G_+)^\ast}w(e)}{\prod_{v\in V(G_+)}w(v)}=\frac{2^{h_2\left(1+\left(\frac{-4}{p}\right)\right)} 
3^{h_3\left(1+\left(\frac{-3}{p}\right)\right)}}{2^{2h_2}3^{2h_3}}=\frac{1}{2^{n_2}3^{n_3}},
$$
where 
$$
n_2=\frac{1}{2}\prod_{\ell |d}\left(1-\left(\frac{-4}{\ell}\right)\right),
\qquad 
n_3=\frac{1}{2}\prod_{\ell |d}\left(1-\left(\frac{-3}{\ell}\right)\right). 
$$
We get 
\begin{equation}\label{eqJL3}
|\Phi_{J^d, p}|=\frac{|P(-p-1)P'(p+1)|}{2m(d')2^{n_2}3^{n_3}}=\frac{p+1}{m(d')2^{n_2}3^{n_3}}\prod_{i=2}^{h(d')} (p+1-\la_i)(p+1+\la_i). 
\end{equation}

\begin{rem}\label{lastrem}
Note that our formulas (\ref{eqJL1}), (\ref{eqJL2}) and (\ref{eqJL3}) differ slightly from the formula 
in Theorem 2.3 of \cite{JLComGrp}. The formula in \cite{JLComGrp} is not correct as stated, which is a result 
of a subtraction mistake on line 2 of page 235 in \cite{JLComGrp}. 
\end{rem}



\end{document}